\documentclass[leqno,11pt]{amsart}

\usepackage{amsmath}
\usepackage{graphicx, color}
\usepackage{amscd}
\usepackage{amsfonts}
\usepackage{amssymb}
\usepackage{mathrsfs}
\usepackage{mathtools}

\newtheorem{thm}{Theorem}[section]

\newtheorem{lem}[thm]{Lemma}
\newtheorem{prop}[thm]{Proposition}

\newtheorem{rem}{Remark}[section]

\numberwithin{equation}{section}

\newcommand\be{\begin{equation}}
\newcommand\ee{\end{equation}}
\newcommand\R{\mathbb R}
\newcommand\N{\mathbb N}
\newcommand\rn{\mathbb R^n}
\newcommand\Rnp{\mathbb R^n_+}
\newcommand\rnp{\mathbb R^n_+}

\newcommand{\De}{\Delta}
\allowdisplaybreaks
\def\eps{\varepsilon}

\title[Liouville-type theorems in half-spaces]
{Liouville-type theorems for unbounded solutions of elliptic equations in half-spaces}

\author[Sirakov]{Boyan Sirakov}
\address{PUC-Rio, Departamento de Matematica \\
Rua Marqu\^es de S\~ao Vicente 225, \\
G\'avea, Rio de Janeiro - CEP 22451-900, Brazil}
\email{bsirakov@mat.puc-rio.br}

\author[Souplet]{Philippe Souplet}%
\address{Universit\'e Sorbonne Paris Nord,
CNRS UMR 7539, Laboratoire Analyse, G\'{e}om\'{e}trie et Applications,
93430 Villetaneuse, France}
\email{souplet@math.univ-paris13.fr}

\begin{document}

\begin{abstract} We prove that the Dirichlet problem for the Lane-Emden equation in a half-space has no positive solutions which grow
at most like
the distance to the boundary to a power given by the natural scaling exponent of the equation;
in other words, we rule out {\it type~I grow-up} solutions.
Such a nonexistence result was previously available only for bounded solutions,
or under a restriction on the power in the nonlinearity.
Instrumental in the proof are local pointwise bounds for the logarithmic gradient of the solution and its normal derivative, which we also establish.

\smallskip
{\bf Keywords.} Lane-Emden equation, semilinear elliptic equation, half-space, Liouville-type theorem,
unbounded solutions, logarithmic gradient estimate, Bernstein method.
\end{abstract}

\maketitle

\tableofcontents

\section{Introduction and main results}

In this paper we prove a Liouville type theorem for positive solutions of the Lane-Emden equation in a half-space, with a Dirichlet boundary condition.

Similarly to the original Liouville result on bounded harmonic functions, a Liouville type theorem states that a given PDE has no nontrivial solutions, and in most cases is restricted to signed solutions in the Euclidean space or some unbounded domain in that space. Arguably the most outstanding theorem of this type for semilinear elliptic equations was obtained in \cite{GS}, on the problem
\begin{equation}\label{eq-gs}
-\De u = u^p , \qquad u>0,
\end{equation}
where $p>1$. Gidas and Spruck proved that there do not exist classical
solutions of \eqref{eq-gs} in $\R^n$ provided$1<p<p_c:=(n+2)/(n-2)_+$
(see also \cite{ChL}, \cite{BVV} or \cite{QSbook} for
proofs of this result). If $p\ge p_c$ there are (bounded) positive solutions of \eqref{eq-gs}.

Equation \eqref{eq-gs}, usually referred to as the Lane-Emden equation, is generally viewed as the simplest and most representative semilinear elliptic equation, and as such has been the object of an enormous number of theoretical studies. For an extended and up-to-date list of references we refer to the recent book \cite{QSbook}. The Lane-Emden equation is also the base model for many more general equations, in terms of the elliptic operator in the left-hand side or the nonlinear function in the right-hand side of \eqref{eq-gs}.

A full answer to the existence question for \eqref{eq-gs} is currently not available for proper subdomains of $\rn$. This is so even for the {\it Dirichlet problem in the half-space} $\Rnp=\{x\in\R^n;\,x_n>0\}$,
despite the large number of works on that problem. Studying \eqref{eq-gs} in a half-space is important both because the half-space is the simplest unbounded domain with unbounded boundary and because performing a blow-up close to the boundary for general equations in a smooth domain leads to \eqref{eq-gs} in a half-space.

The latter observation, together with degree theory, was used by Gidas and Spruck in \cite{GS2} to prove existence results for a large class of elliptic equations in a smooth bounded domain. They proved that \eqref{eq-gs} with Dirichlet boundary condition has no solutions in a half-space provided $1<p\le p_c$.
The proof in \cite{GS2} uses a moving planes argument (combined with Kelvin transform), which reduces the problem to the one-dimensional case.

A few years later Dancer \cite{D} proved by the method of 
moving planes that bounded solutions of \eqref{eq-gs} with Dirichlet boundary condition in a half-space are monotone in
the normal direction, and deduced that a nontrivial bounded solution in $\rnp$ gives rise to a solution in $\R^{n-1}$; this implies
that the problem has no {\it bounded} solutions in the range $1<p<p_D:=(n+1)/(n-3)_+$.
A further development was made by Farina \cite{F}, who used variational estimates and stability to show that bounded (and even stable outside a compact set in a slightly smaller range for $p$)
solutions do not exist if $1<p<p_F(n):=(n^2-10n+8\sqrt{n}+13)/((n-3)(n-11)_+)$.
The most general results in these lines of research, as well as extensive list of references, can be found in the recent work \cite{DF}.

The question of existence of bounded solutions of \eqref{eq-gs} in a half-space was fully answered a few years ago by Chen, Lin and Zou \cite{CLZ}, who proved there are no such solutions for any $1<p<\infty$. In that paper the authors used a well-chosen auxiliary function involving derivatives of $u$, as well as convexity considerations.

To our knowledge, non-existence of unbounded solutions of \eqref{eq-gs} in $\rnp$ is completely open
in the supercritical range $p>p_{c}$
 (and for unstable solutions for $p\ge p_F(n+1)$). This is the problem we study here, and prove that {\it for any} $p>1$
 there are no solutions with controlled (and natural) growth.

\begin{thm}
\label{mainthmup}
Let $p>1$. Then the problem
\be\label{pbmDirichletHalfspace1}
\left\{\begin{array}{llll}
\hfill -\Delta u&=&u^p,& x\in\Rnp
\vspace{1mm} \\
\hfill u&=&0,& x\in\partial\Rnp
\end{array}
\right.
\ee
does not admit any positive classical solution satisfying the growth condition
\be\label{growthup}
u(x)\le C(1+x_n)^{2/(p-1)}
\ee
for some $C>0$.
\end{thm}

We note that the exponent $2/(p-1)$ is the natural scaling exponent of the Lane-Emden equation, in the sense that whenever $u(x)$ is a solution, $t^{2/(p-1)}u(tx)$ is a solution too, for each $t>0$.
By analogy with the well-known terminology for finite time blow-up solutions of the parabolic equation corresponding to \eqref{eq-gs}
(cf.~\cite{MM} and see also the references in \cite{QSbook}),
condition \eqref{growthup} can be viewed as a {\it type~I grow-up} assumption.
 In other words, Theorem~\ref{mainthmup} rules out type~I grow-up solutions of \eqref{pbmDirichletHalfspace1},
 and such kind of result seems to be new in an elliptic context.
It is interesting and worth noting that when $1<p<p_c$, every positive solution of the Lane-Emden equation in an arbitrary unbounded domain {\it decreases} away from the boundary like the distance to the boundary to the negative power $-2/(p-1)$;~see  \cite{D2, PQS}. Such a property is of course false for larger $p$'s. Theorem \ref{mainthmup} shows that in a half-space solutions cannot {\it grow}  faster than
the distance to the boundary to the power $+2/(p-1)$, for arbitrary $p$.

Theorem \ref{mainthmup} is a special case of the following non-existence result for more general nonlinearities.

\begin{thm}
\label{mainthm7}
Let $p>1$ and let $f:[0,\infty)\to [0,\infty)$ be such that $f(0)=0$, $f\not\equiv0$, $f\in C^1([0,\infty))\cap C^2(0,\infty)$
 and $f''$ is locally H\"older continuous on $(0,\infty)$. Assume that, for some $C>0$,
\be\label{hypfmain}
0\le sf''(s)\le C(1+s)^{p-1},\quad s>0.
\ee
If $u$ is a nonnegative classical solution of the problem
\be\label{pbmDirichletHalfspace2}
\left\{\begin{array}{llll}
\hfill -\Delta u&=&f(u),& x\in\Rnp
\vspace{1mm} \\
\hfill u&=&0,& x\in\partial\Rnp
\end{array}
\right.
\ee
which  satisfies the growth condition \eqref{growthup},
then  $u\equiv0$.
\end{thm}

The convexity of $f$ is an essential hypothesis in this theorem; apart from that, \eqref{hypfmain} is
mostly a restriction on the growth of $f(s)$ at infinity.

The main point in the proof of Theorem~\ref{mainthm7} is to show that  $u$ is convex in the $x_n$-direction, as was already observed in \cite{CLZ} for bounded solutions. Here we will further develop the simplified approach to the result in \cite{CLZ} given in \cite{QSbook}, which shows that
a suitable perturbation of the auxiliary function
$$\xi=\frac{u_{x_nx_n}}{(1+x_n)u_{x_n}}$$
is a supersolution of the elliptic inequality \eqref{S-CLZeqnxieps0} below, to which the maximum principle applies.

The main novelties in our proofs are the following. First, we observe that it is enough to have control on the growth of $\xi$ and of the
drift coefficient of \eqref{S-CLZeqnxieps0} in order to infer a maximum principle; second and most importantly, we obtain such control.
The crucial quantity to study turns out to be the {\it logarithmic gradient of} $u_{x_n}$, which is involved in both $\xi$ and
the drift coefficients of \eqref{S-CLZeqnxieps0}.
In the next section we state some new, and essentially optimal, local estimates on the logarithmic gradients of $u$ and $u_{x_n}$, whenever these functions are positive, for a general semilinear elliptic equation.
These functional inequalities evaluate locally $|\nabla u|/u$ and $|\nabla u_{x_n}|/u_{x_n}$ in terms of bounds on $f$ and its derivatives on the range of $u$, making also explicit the role played by the positive and negative parts of $f$ and $f^\prime$.

We note that in the previous works \cite{CLZ, QSbook} on bounded solutions of \eqref{pbmDirichletHalfspace2}, key estimates on the auxiliary function $\xi$ were directly inferred from the classical Harnack inequality. This approach no longer applies for unbounded solutions and we have instead to resort to our local estimates on the logarithmic gradients, which serve as a quantitative refinement of the Harnack inequality.

While instrumental in the proofs of the theorems above, the estimates in the next section are important in their own right; further applications will be given in a forthcoming paper.

\section{Local estimates on logarithmic gradients}

We consider  positive solutions of the general equation
\be\label{eqnf}
-\Delta u=f(u)
\ee
in an arbitrary domain, where $f$ is a differentiable function.

Throughout the rest of the article, we use the following notation for balls and strips
$$B_R:= \{x\in\R^n:\ |x|<R\},\qquad \Sigma_R:=\{x\in\R^n:\ 0<x_n<R\},\quad R>0.$$
Also we use the standard notation $t=t_+-t_-$, $|t|= t_++t_-$, for $t\in \R$.

The first goal of this section is to give a pointwise local estimate on the logarithmic gradient of $u$, i.e.~the ratio $\frac{|\nabla u|}{u}$.
 Moreover, we shall see that this estimate is optimal in general.

\begin{thm}
\label{mainthm1}
Let $f:(0,\infty)\to \R$ be a $C^1$ function, with $f'$ locally H\"older continuous.
Then any positive classical solution $u$ of \eqref{eqnf} in $\overline B_1$
satisfies the estimate
$$\sup_{B_\sigma}\frac{|\nabla u|}{u}\le
C_1\biggl\{1+\sup_{B_1}\sqrt{[f'(u)]_++\frac{f_-(u)}{u}}\biggr\},$$
for any $\sigma\in (0,1)$,
where the positive constant $C_1$ depends only on $n,\sigma$.
\smallskip

In particular, for $f(u)=u^p$ with $p>1$, we have
$$\sup_{B_\sigma}\frac{|\nabla u|}{u}\le
C_2\Bigl(1+\sup_{B_1}u\Bigr)^{\frac{p-1}{2}},$$
where the positive constant $C_2$ depends only on $n,\sigma,p$.
\end{thm}

The following proposition shows the optimality
of the last estimate, for supercritical power nonlinearities (which are the main motivation of this paper).

\begin{prop}
\label{propoptim}
Let $n\ge 3$ and $p\ge (n+2)/(n-2)$. There exist a constant $c>0$ and a sequence of classical solutions of \eqref{eq-gs} on $B_1$ such that
$\sup_{B_1}u_j\to\infty$ as $j\to\infty$ and, for each $\sigma\in (0,1)$,
$$\sup_{B_\sigma}\frac{|\nabla u_j|}{u_j}\ge
c\Bigl(\sup_{B_1}u_j\Bigr)^{\frac{p-1}{2}},\quad j\to\infty.$$
\end{prop}

Pointwise estimates of the gradient (as opposed to the logarithmic gradient) of the bounded global solutions of a semilinear elliptic equation are extensively studied, starting with the classical work by Modica \cite{M}. General results of Modica type and many more references can be found in \cite{CD}. A global gradient estimate for bounded solutions of the Dirichlet problem for the Lane-Emden equation in a half-space in the form
$$
\sup_{\rnp} \left( |\nabla u|^2 + u^{p+1}\right)\le (\sup_{\rnp} u)^{p+1}
$$ was proved by Farina and Valdinocci \cite{FV}.

Much less is known on estimates of $|\nabla u|/u$, especially in the supercritical range $p\ge p_c$. 
Li~\cite{Li} proved a global estimate for the Lane-Emden equation on manifolds, under the (restrictive) hypothesis $p<n/(n-2)$.
We observe that for $1<p<p_c$, as a consequence of \cite{GS} (see also \cite{BVV, PQS}), any positive classical solution of the 
Lane-Emden equation in $B_1$ is universally bounded in $B_\sigma$ for any $\sigma\in(0,1)$, 
so by a simple argument based on Harnack's inequality and elliptic regularity
(see for instance \cite[p.~56]{QSbook} or the end of the proof of Lemma \ref{lembdry} below),
we have the universal estimate $\sup_{B_\sigma} |\nabla u|/u\le C(n,p,\sigma)$.
Such a universal estimate fails for $p\ge p_c$.
An estimate in the form $|\nabla u|/u \le C u^{-a}$ for some universal $a>0$ was proved in the recent work~\cite{BGV} 
for the equation $-\Delta u = u^p|\nabla u|^q$, provided $p>1$, $0\le q<2$ and $p+q<(n+3)/(n-1)$.
Differently from 
these works, the estimate in Theorem~\ref{mainthm1} is valid for any $p>1$, as well as for arbitrary nonlinearities $f(u)$.

The proof of Theorem \ref{mainthm1} is based on a  local Bernstein-type argument.
We refer to \cite{Ber,La58,Se69,LY75,Lions85} for classical references on the Bernstein method
and, for more recent developments, to \cite{BGV} and the references in \cite{QSbook}.
\smallskip

The next theorem  is a local estimate for the logarithmic gradient of the derivative $u_{x_n}$, provided this derivative is positive. It plays a key role  in the proof of our Liouville-type theorems.

\begin{thm}
\label{mainthm4}
Let $f:(0,\infty)\to \R$ be a $C^2$ function, with $f''$ locally H\"older continuous,
and let $u$ be a positive classical solution of
$$-\Delta u=f(u),\quad x\in \overline B_1$$
such that
$$u_{x_n}>0\quad\hbox{ in $B_1$.}$$
Then we have the estimate
$$\sup_{B_\sigma}\frac{|\nabla u_{x_n}|}{u_{x_n}}\le
C_1\biggl\{1+
\sup_{B_1}|uf''(u)|^{1/3}\Bigl(1+\sup_{B_1}\Bigl({[f'(u)]_++\frac{f_-(u)}{u}}\Bigr)^{1/6}\Bigr)
+\sup_{B_1}[f'(u)]_-^{1/2}\biggr\},$$
for any $\sigma\in (0,1)$,
where the positive constant $C_1$ depends only on $n,\sigma$.

\smallskip

In particular, for $f(u)=u^p$ with $p>1$, we have
$$\sup_{B_\sigma}\frac{|\nabla u_{x_n}|}{u_{x_n}}\le
C_2\Bigl(1+\sup_{B_1}u\Bigr)^{\frac{p-1}{2}},$$
where the positive constant $C_2$ depends only on $n,\sigma,p$.
\end{thm}

The proof of this theorem uses a similar argument as the proof of Theorem \ref{mainthm1}, and Theorem \ref{mainthm1} itself.

Theorem \ref{mainthm4} has a global extension for Dirichlet problems in half-spaces, of the type considered in the introduction.

\begin{thm}
\label{mainthm6}
Let $p>1$ and let $f$ be as in Theorem~\ref{mainthm7}. 
Let $u$ be a positive classical solution of
\be\label{pbmDirichletHalfspace3}
\left\{\begin{array}{llll}
\hfill -\Delta u&=&f(u),& x\in\Rnp
\vspace{1mm} \\
\hfill u&=&0,& x\in\partial\Rnp,
\end{array}
\right.
\ee
such that $\sup_{\Sigma_R}u<\infty$, for each $R>0$.
Then $u_{x_n}>0$ in $\overline{\rnp}$ and
\be\label{loggradestim}
\sup_{\Sigma_R}\frac{|\nabla u_{x_n}|}{u_{x_n}}\le C_1+C_2 \Bigl(\:\sup_{\Sigma_{R+1}}  u\Bigr)^{\frac{p-1}{2}},\quad R>1.
\ee
Here  $C_1>0$ depends only on $n,f$ and $\sup_{\Sigma_2}u$; while $C_2>0$ depends only on $n,f$.
\end{thm}

\begin{rem}\label{rempos} That positive solutions of \eqref{pbmDirichletHalfspace3} which are bounded on finite strips are strictly monotone in the $x_n$-direction is a general fact, valid for every locally Lipschitz $f$ on $[0,\infty)$ 
such that $f(0)\ge0$. This follows from the Hopf lemma and the method of moving planes, in the form used in \cite{BCNAnn}. A full argument can be found in \cite[Theorem 3.1]{F2} or in the proof of
\cite[Theorem 3.1]{QScpde}. When the space dimension $n=2$, monotonicity holds even without the assumption of boundedness on strips, see \cite{DS}, \cite{FS}.
\end{rem}

\section{Proof of Theorem~\ref{mainthm1} and Proposition~\ref{propoptim}}

We aim to use a  local Bernstein-type argument. Specifically, we will apply the maximum principle to a suitable cut-off of a differential inequality satisfied by $|\nabla \log u|^2$.

The following auxiliary lemma will be used in the proofs of both Theorem~\ref{mainthm1} and Theorem~\ref{mainthm4}.
It is a maximum principle for a specific equation,
and takes care of the cut-off step of the Bernstein procedure.

\begin{lem}
\label{lemcutoff}
Let $A$ be a $C^2$ vector field on $B_1$ and set $z=|A|^2$.
Let $K,\lambda,r>0$ and assume that
\be\label{ineqlemell}
-\frac12\Delta z-A\cdot \nabla z\le K-\lambda z^2\quad\hbox{ in $B_{r}$.}
\ee
Then, for each $\rho\in (0,r)$ there exists a constant $c_1=c_1(n,\lambda,r,\rho)>0$ such that
$$\sup_{B_{\rho}}z\le  \Bigl[\frac{2K}{\lambda}+c_1\Bigr]^{1/2}.$$
\end{lem}

\begin{proof}
Let $\eta\in C^2(\R^n)$ be such that
$\eta=1$ for $|x|\le \rho$, $\eta=0$ for $|x|\ge{\bar r}:=(r+\rho)/2$, and $0<\eta\le 1$ for $\rho<|x|<{\bar r}$.
For $\alpha\in(0,1)$ to be chosen below, assume also that
\be\label{controlcutoff0}
\frac{|\nabla \eta|^2}{\eta}+|\Delta \eta|\le C\eta^\alpha\quad\hbox{ in $B_{\bar r}$,}
\ee
where $C$ depends on $\alpha,n,r,\rho$. Such a choice of the cut-off $\eta$ is possible
-- take for instance $\eta(t) = ({\bar r}-t)_+^{2+\beta}$ in a neighborhood of ${\bar r}$, with $\beta = 2\alpha/(1-\alpha)>0$,
and similarly $\eta(t) = 1- (t-\rho)_+^{2+\beta}$ in a neighborhood of $\rho$.

Let $\phi=\eta z$. We have
$$\nabla\phi=\eta\nabla z+z\nabla\eta
\quad\hbox{ and }\quad\Delta\phi=\eta\Delta z+z\Delta\eta+2\nabla\eta\cdot\nabla z.$$
Setting $\mathcal{L}\phi=-\frac12\Delta \phi-A\cdot \nabla \phi$, we get
$$\begin{aligned}
\mathcal{L}\phi
&=\eta\mathcal{L}z-\frac12z\Delta\eta-zA\cdot \nabla\eta-\nabla \eta\cdot\nabla z\\
&=\eta\mathcal{L}z-\frac12z\Delta\eta-zA\cdot \nabla\eta-\frac{\nabla\eta}{\eta}\cdot(\nabla\phi-z\nabla \eta)
\quad\hbox{ in $B_{\bar r}$.}
\end{aligned}$$
Consequently, by using \eqref{controlcutoff0} and $z=|A|^2$ we obtain
$$\begin{aligned}
\mathcal{L}_1\phi
&:=\mathcal{L}\phi+\frac{\nabla\eta}{\eta}\cdot\nabla\phi \\
&=\eta\mathcal{L}z+\Bigl(\frac{|\nabla\eta|^2}{\eta}-\frac12{\Delta\eta}\Bigr)z-zA\cdot \nabla\eta\\
&\le\eta\mathcal{L}z+C\eta^\alpha z+C\eta^{(1+\alpha)/2}z^{3/2}
\quad\hbox{ in $B_{\bar r}$.}
\end{aligned}$$
Now choose $\alpha=1/2$. By using \eqref{ineqlemell} and Young's inequality $ab\le \varepsilon a^q + C_\varepsilon b^{q/(q-1)}$,
with $q=2$ and $q=4/3$, respectively, we obtain,
for each $\varepsilon\in(0,\lambda)$,
$$\begin{aligned}
\mathcal{L}_1\phi
&\le \eta (K-\lambda z^2)+C(n,r,\rho)(\eta^{1/2} z + \eta^{3/4} z^{3/2}) \\
&\le K - (\lambda-\varepsilon)\eta z^2 + C(\eps,n,r,\rho)   \\
&\le K - (\lambda-\varepsilon)\phi^2 + C(\eps,n,r,\rho)   \quad\hbox{ in $B_{\bar r}$,}
\end{aligned}$$
where we also used $0\le\eta\le 1$.
Set $\varepsilon={\lambda}/2$. Since $\phi=0$ for $|x|\ge{{\bar r}}$,
$\phi$ attains an interior maximum at some $x_0$ with $|x_0|<{{\bar r}}$.
At that point we have
$$0\le \mathcal{L}_1\phi(x_0)\le K - \frac{\lambda}{2}\phi^2 + C(\lambda,n,r,\rho),
$$
hence $\phi(x_0)\le [\frac{2}{\lambda}(K+C(\lambda,n,r,\rho))]^{1/2}$.
Since $\eta=1$ for $|x|\le\rho$, we deduce that
$$\sup_{B_{\rho}}z\le \Bigl[\frac{2}{\lambda}(K+C(\lambda,n,r,\rho))\Bigr]^{1/2},$$
and we are done.
\end{proof}

\begin{proof} [Proof of Theorem~\ref{mainthm1}]
We look for an equation satisfied by $\frac{|\nabla u|^2}{u^2}$. To this end, we first set
$$v=\log u,$$
i.e. $u=e^v$. By elliptic regularity $v\in C^3(B_1)$.
We compute
$$-\Delta v=-\nabla\cdot\frac{\nabla u}{u}=-\frac{\Delta u}{u}+\frac{|\nabla u|^2}{u^2}=\frac{f(u)}{u}+|\nabla v|^2.$$
Setting
$$h(v):=e^{-v}{f(e^v)}=  f(u)/u,$$
we see that $v$ solves
\be\label{eqv}
-\Delta v=h(v)+|\nabla v|^2.
\ee
Let now $z:=|\nabla v|^2$.
Using
\be\label{eqBochner}
\Delta |\nabla v|^2=2\nabla (\Delta v)\cdot \nabla v+2|D^2v|^2,
\quad\hbox{ where } |D^2v|^2:=\sum_{1\le i,j\le n}(v_{x_ix_j})^2,
\ee
it follows that
\be\label{eqz1}
-\frac12\Delta z=h'(v)z+\nabla v\cdot \nabla z-|D^2v|^2.
\ee
On the other hand, as a consequence of the Cauchy-Schwarz inequality, we have
\be\label{eqCS}
|\Delta v|^2\le n|D^2v|^2.
\ee
Therefore, by \eqref{eqv} and \eqref{eqz1}, we obtain
\be\label{eqz2}
\mathcal{L}z:=-\frac12\Delta z-\nabla v\cdot \nabla z\le h'(v)z-\frac{1}{n}(h(v)+z)^2.
\ee
Observe that
$$h'(v)=f'(e^v)-e^{-v}{f(e^v)}=f'(u)-h(v).$$
Hence
$$
\begin{aligned}
\mathcal{L}z
&\le f'(u)z - \left(1+\textstyle\frac{2}{n}\right) h(v)z - \frac{1}{n}h(v)^2-\frac{1}{n}z^2 \\
&\le \left( [f'(u)]_+ + \left(1+\textstyle\frac{2}{n}\right)h_-(v)\right)z-\frac{1}{n}z^2\\
&=   \left( [f'(u)]_+ + \Bigl(1+\frac{2}{n}\Bigr)\frac{f_-(u)}{u}\right)z-\frac{1}{n}z^2.
\end{aligned}
$$
Now denote
\be\label{defM}
K=\frac{n}{2} \sup_{x\in B_1} \Bigl[[f'(u(x))]_++\Bigl(1+\frac{2}{n}\Bigr)\frac{f_-(u(x))}{u(x)}\Bigr]^2,
\ee
assuming this quantity is finite (else there is nothing to prove).
Using $ab\le \frac{na^2}{2}+\frac{b^2}{2n}$, we get
$$\mathcal{L}z\le K-\frac{1}{2n}z^2\quad\hbox{in $B_1$,}$$
and the conclusion follows from Lemma~\ref{lemcutoff}.
\end{proof}

\begin{proof} [Proof of Proposition~\ref{propoptim}]
It is well known (see~e.g.~\cite[Section~9]{QSbook}) that for $n\ge 3$ and $p\ge (n+2)/(n-2)$, there exists a
classical solution of \eqref{eq-gs} in $\R^n$ which is a radial decreasing function, namely, $u(x)=U(|x|)$, with $U'<0$ in $(0,\infty)$. 

Set $u_j(x)=j^{2/(p-1)}u(jx)$ for $j\in\N^*$. Then $u_j$ is a classical solution of \eqref{eq-gs} in $B_1$, which satisfies
\be\label{supB1uj}
\sup_{B_1}u_j=U(0)j^{2/(p-1)},
\ee
and
\be\label{nablauj}
\frac{|\nabla u_j(x)|}{u_j(x)}=j\frac{|U'(j|x|)|}{U(j|x|)}.
\ee
We take $x=j^{-1}e_1$ in \eqref{nablauj}.
For any fixed $\sigma\in (0,1)$,  by using \eqref{supB1uj} we get for each $j>\sigma^{-1}$,
$$\sup_{B_\sigma}\frac{|\nabla u_j|}{u_j}\ge \frac{|\nabla u_j(j^{-1}e_1)|}{u_j(j^{-1}e_1)}
=j\frac{|U'(1)|}{U(1)}=\frac{|U'(1)|}{U(1)}(U(0))^{\frac{1-p}{2}}\Bigl(\sup_{B_1}u_j\Bigr)^{\frac{p-1}{2}}.$$
This proves the proposition.
\end{proof}

\section{Proof of Theorems~\ref{mainthm4} and \ref{mainthm6}}

We again apply a Bernstein-type argument, this time on the equation satisfied by $u_{x_n}$.

\begin{proof} [Proof of Theorem~\ref{mainthm4}]
We look for an equation satisfied by $\frac{|\nabla u_{x_n}|^2}{(u_{x_n})^2}$.
We first note that the function $V=u_{x_n}$ satisfies
$$-\Delta V=f'(u)V.$$
Next set
$$W=\log V,$$
i.e. $V=e^W$.
We compute
$$-\Delta W=-\nabla\cdot\frac{\nabla V}{V}=-\frac{\Delta V}{V}+\frac{|\nabla V|^2}{V^2},$$
hence
\be\label{eqW}
-\Delta W=f'(u)+|\nabla W|^2.
\ee
Let now $Z:=|\nabla W|^2$. By elliptic regularity $u\in C^4(B_1)$, hence $W\in C^3(B_1)$.
Using the analogue of formula \eqref{eqBochner}, it follows that
\be\label{eqZ1}
-\frac12\Delta Z=f''(u)\nabla u \cdot \nabla W+\nabla W \cdot \nabla Z-|D^2W|^2.
\ee
Therefore, by the analogue of \eqref{eqCS},
we obtain
$$
\tilde{\mathcal{L}}Z:=-\frac12\Delta Z-\nabla W\cdot \nabla Z\le f''(u)\nabla u\cdot \nabla W-\frac{1}{n}(f'(u)+Z)^2.
$$
Let $\tilde\sigma= (1+\sigma)/2$.
Next, using $(a+b)^2\ge (3/4) a^2 - 3b^2$ , Young's inequality and Theorem~\ref{mainthm1} which we already proved, we get
$$
\begin{aligned}
\tilde{\mathcal{L}}Z
&\le \frac{1}{2n}|\nabla W|^4+C(n)|f''(u)\nabla u|^{4/3}-\frac{1}{n}(f'(u)+Z)^2 \\
&\le \frac{1}{2n}Z^2+C(n,\sigma)|uf''(u)|^{4/3}\Bigl\{1+\sup_{B_1}\Bigl({[f'(u)]_++\frac{f_-(u)}{u}}\Bigr)^{2/3}\Bigr\}
 -\frac{3}{4n}Z^2+\frac{3}{n}[f'(u)]_-^2,
\end{aligned}
$$
for all $x\in B_{\tilde\sigma}$. Hence
$$\tilde{\mathcal{L}}Z
\le -\frac{1}{4n}Z^2+C{(n,\sigma)}\tilde M \quad\hbox{ in $B_{\tilde\sigma}$},$$
where
$$\tilde M:=\sup_{B_1}|uf''(u)|^{4/3}\Bigl\{1+\sup_{B_1}\Bigl({[f'(u)]_++\frac{f_-(u)}{u}}\Bigr)^{2/3}\Bigr\}
+\sup_{B_1}[f'(u)]_-^2,$$
which can be assumed to be finite.
The conclusion now follows from Lemma~\ref{lemcutoff} applied to~$Z$.
\end{proof}

We now turn to the proof of Theorem~\ref{mainthm6}.
The estimate of the ratio $|\nabla u_{x_n}|/u_{x_n}$ away from the boundary, say, for $x_n>1$, is a
consequence of Theorem~\ref{mainthm4}.
The latter is not directly applicable to the strip $0<x_n\le 1$ but,
since $u$ is bounded at finite distance from the boundary by assumption,
the estimate in that region will follow by applying Harnack's inequality to the equation satisfied by $u_{x_n}$.
This is the contents of the following lemma.

\begin{lem}
\label{lembdry}
Let $f:[0,\infty)\to \R$ be a function of class $C^1$ on $[0,\infty)$,
and assume $f(0)=0$.
Let $u\ge 0$ be a classical solution of
\be\label{pbmDirichletStrip}
\left\{\begin{array}{llll}
\hfill -\Delta u&=&f(u),& x\in\Sigma_2
\vspace{1mm} \\
\hfill u&=&0,& x\in\partial\Rnp
\end{array}
\right.
\ee
such that $M:=\sup_{\Sigma_2}u<\infty$ and
$$u_{x_n}>0\quad\hbox{ in }\overline \Sigma_2.$$
Then there exists a constant $C=C(M,n,f)>0$ such that
$$\sup_{\Sigma_1}\frac{|\nabla u_{x_n}|}{u_{x_n}}\le C.$$
\end{lem}

\begin{proof}
Similarly as in \cite[p.~55]{QSbook},
it is convenient to first extend $u$ to a (sign-changing) solution on the
symmetric strip $\tilde\Sigma_2:=\{x\in\R^n;\ -2<x_n<2\}$.
Namely, we extend $f$ to a $C^1$
function on $\R$ by setting $f(s)=-f(-s)$ for $s<0$,
and we extend $u$ to $\tilde\Sigma_2$ by setting
$$u(x',x_n)=-u(x',-x_n),\qquad x'\in \R^{n-1},\ -2<x_n<0.$$
We observe that $u\in C^2(\tilde\Sigma_2)$ and
\be\label{eqfuext}
-\Delta u=f(u),\qquad x\in \tilde\Sigma_2.
\ee
Indeed,  by elliptic regularity,
we have $u\in C^2(\Sigma_2\cup\partial\Rnp)$.
Let $\partial_\alpha$ be a first or second order partial derivative
and let $\ell\in\{0,1,2\}$ be the number of occurrences of the variable $x_n$ in $\partial_\alpha$.
Clearly, if $\ell=1$,
then $\partial_\alpha u$ has the same limits from both sides of $\partial \Rnp$.
Next, if $\ell=0$, then $\partial_\alpha u=0$ on $\partial \Rnp$ due to $u=0$ on $\partial \Rnp$,
so that $\partial_\alpha u$ also has the same limits from both sides of $\partial \Rnp$.
Using $-\Delta u=f(u)$ and $f(0)=0$, it follows in particular that
$$u_{x_nx_n}=0\quad\hbox{ on $\partial \Rnp$.}$$
Therefore,
for $\ell=2$, $\partial_\alpha u=u_{x_nx_n}$ again has the same limits from both sides of $\partial \Rnp$.
It follows that $u\in C^2(\tilde\Sigma_2)$
and that \eqref{eqfuext} is satisfied.

Now the function $v:=u_{x_n}$ is a strong
solution of
\be\label{tuk}
-\Delta v=f'(u)v,\quad x\in\tilde\Sigma_2,\ee
along with $v>0$.
Since $M_1:=\sup_{\tilde\Sigma_2}|f'(u)|<\infty$ by our assumption,
we deduce from Harnack's inequality
(see \cite[Theorem 8.20 and Corollary 8.21]{GT98})
that there exists a constant $K=K(M_1,n)>0$ such that
$$\sup_{y\in B_{1/2}(x)} v(y)\le Kv(x),\quad x\in \tilde\Sigma_1.$$
Consequently, by standard elliptic estimates and \eqref{tuk}, we deduce that
$$|\nabla v(x)|\le C(n)\sup_{y\in B_{1/2}(x)} |\Delta v(y)|+\sup_{y\in B_{1/2}(x)}  v(y)\le C(n)(1+M_1)Kv(x),\quad x\in \tilde\Sigma_1,$$
and the conclusion follows.
\end{proof}

\begin{proof} [Proof of Theorem~\ref{mainthm6}]
Let $R>1$ and pick $x\in\Sigma_R$.
If $0<x_n\le 1$, then
$$\frac{|\nabla u_{x_n}(x)|}{u_{x_n}(x)}\le C\Bigl(n,f,\sup_{\Sigma_2}u\Bigr),$$
by Lemma~\ref{lembdry}.
Let us thus consider the case $1<x_n<R$.
Note that, as a consequence of $f^\prime\in C([0,\infty))$ and \eqref{hypfmain}, we have
\be\label{hypfmain2}
f'(s)\le C(1+s)^{p-1},\quad s>0.
\ee
Moreover we have $f'\ge 0$ due to $f\ge 0$, $f(0)=0$ and $f''\ge 0$.
Applying Theorem~\ref{mainthm4} in the ball $B_1(x)\subset\Sigma_{R+1}$
and using  \eqref{hypfmain}, \eqref{hypfmain2}; $f, f'\ge 0$ (and the invariance of the equation by translation), we obtain
$$\begin{aligned}
\frac{|\nabla u_{x_n}(x)|}{u_{x_n}(x)}
&\le C(n)\biggl\{1+\sup_{\Sigma_{R+1}}|uf''(u)|^{1/3}\Bigl(1+\sup_{\Sigma_{R+1}}\bigl(f'(u)\bigr)^{1/6}\Bigr)\biggr\} \\
&\le C_2(n,f)\biggl\{1+\Bigl(\sup_{\Sigma_{R+1}}u\Bigr)^{(p-1)/2}\biggr\}.
\end{aligned}$$
The result follows.
\end{proof}

\section{Proof of Theorems \ref{mainthmup} and \ref{mainthm7}}

Theorem \ref{mainthmup} is a special case of Theorem \ref{mainthm7}.
We shall prove Theorem \ref{mainthm7} by combining the estimates in Theorem~\ref{mainthm6} with a suitable modification and
extension of the arguments in \cite{CLZ} and \cite{QSbook}.

We recall that $v:=u_{x_n}>0$ in $\overline{\Rnp}$, see Remark \ref{rempos}.
We set
$$w:=u_{x_nx_n},\quad z=(1+x_n)v,\quad x\in \R^n,$$
and define the key auxiliary function
\be\label{def-xi}
\xi:=\frac{w}{z}=\frac{u_{x_nx_n}}{(1+x_n)u_{x_n}},\qquad x\in \overline{\Rnp}.
\ee

We first claim that $\xi$ satisfies
\be\label{S-CLZeqnxi}
-\Delta\xi
\ge 2\xi^2+2\frac{\nabla z}{z}\cdot\nabla\xi,\qquad x\in \Rnp.
\ee
Indeed, we have
$$-\Delta v=f'(u)v,\quad -\Delta w=f'(u)w+f''(u)v^2$$
and
$$-\Delta z=-(1+x_n)\Delta v-2v_{x_n}=f'(u)z-2w.$$
Using the formula
$$-\Delta\Bigl(\frac{w}{z}\Bigr)
=-\nabla\cdot\frac{z\nabla w-w\nabla z}{z^2}=\frac{-z\Delta w+w\Delta z}{z^2}+2\frac{\nabla z}{z}\cdot\nabla\Bigl(\frac{w}{z}\Bigr)$$
we deduce that
$$-\Delta\xi
=\frac{z(f'(u)w+f''(u)v^2)+w(-f'(u)z+2w)}{z^2}+2\frac{\nabla z}{z}\cdot\nabla\xi$$
and, since $f''\ge 0$, inequality \eqref{S-CLZeqnxi} follows.

The idea is to reach a contradiction by applying the maximum principle to the inequality \eqref{S-CLZeqnxi} satisfied by $\xi$.
However, to avoid possible difficulties at  infinity,
we need to consider a perturbation of $\xi$
(note that, unlike in \cite{CLZ}  or \cite{QSbook}, the function $\xi$ need not decay as $x_n\to \infty$ so that its infimum might
not be realized in any strip).
 Specifically, setting
$$\xi_\eps=\xi+\eps \psi,$$
with $\eps>0$ and $\psi$ a $C^2$ function, it follows from \eqref{S-CLZeqnxi} that
\be\label{S-CLZeqnxieps0}
-\Delta\xi_\eps-2\frac{\nabla z}{z}\cdot\nabla\xi_\eps
\ge 2\bigl[\xi_\eps-\eps\psi\bigr]^2-\eps\Bigl\{\Delta\psi+2\frac{\nabla z}{z}\cdot\nabla\psi\Bigr\},\qquad x\in \Rnp.
\ee
At this point we choose
$$\psi(x)=\log(1+|x|^2).$$

\smallskip
We next use Theorem~\ref{mainthm6} and assumption \eqref{growthup} to
estimate the function $\xi$ and the drift coefficient in \eqref{S-CLZeqnxieps0}.
By \eqref{growthup} and \eqref{loggradestim}, we have
\be\label{boundnablavv0}
\frac{|\nabla v(x)|}{v(x)}\le C_1+C_2\biggl[\sup_{\Sigma_{1+x_n}} u\biggr]^{(p-1)/2}\le C_3(1+x_n)
\quad\hbox{ for all $x\in \Rnp$,}
\ee
with $C_3=C_1+C_2C^{(p-1)/2}$.
Recalling \eqref{def-xi}, it follows in particular that the function $\xi$ is bounded, hence
\be\label{liminfty}
\lim_{|x|\to\infty,\ x\in \overline{\Rnp}} \xi_\eps(x)=\infty.
\ee
On the other hand, combining
\be\label{nablazz}
\frac{\nabla z}{z}=\frac{\nabla v}{v}+\frac{e_n}{1+x_n},
\ee
$$|\nabla\psi|=2|x|(1+|x|^2)^{-1}\le (1+|x|^2)^{-1/2},$$
estimate \eqref{boundnablavv0}, and the boundedness of $\Delta\psi$, we see that
\be\label{boundeddrift}
K:=\sup_{\Rnp} \Bigl|\Delta\psi+2\frac{\nabla z}{z}\cdot\nabla\psi\Bigr|<\infty.
\ee

Now assume for contradiction that $\xi(x_0)<0$ for some $x_0\in\Rnp$.
Setting $\sigma=-\frac12 \xi(x_0)>0$ and taking $\eps_0>0$ sufficiently small, we get, since $\xi$ is bounded,
\be\label{contrad1}
-\infty<\inf_{\Rnp}\xi_\eps<-\sigma, \quad\hbox{ for all $\eps\in(0,\eps_0]$.}
\ee
Also, since $u=0$ on $\partial\Rnp$, we have
$\xi_\eps\ge\xi=0$ on $\partial \Rnp$.
In view of \eqref{liminfty}, there exists $x_\eps\in \Rnp$ such that
$\xi_\eps(x_\eps)=\inf_{\Rnp}\xi_\eps$.
Then \eqref{S-CLZeqnxieps0}, \eqref{boundeddrift} and \eqref{contrad1} yield
\be\label{contrad2}
0\ge -\Delta\xi_\eps(x_\eps)-2\frac{\nabla z}{z}\cdot\nabla\xi_\eps(x_\eps)
\ge 2\bigl[\xi_\eps(x_\eps)-\eps\psi(x_\eps)\bigr]^2-K\eps\ge 2\sigma^2-K\eps>0
\ee
for
$$\eps= \min\left\{\eps_0,\frac{\sigma^2}{K}\right\},$$
which is a contradiction.
Therefore $\xi\ge 0$ in $\Rnp$.

We have proved that
$$
u_{x_nx_n}\ge0 \qquad \mbox{in }\;\rnp.
$$

To finish the proof  of Theorem~\ref{mainthm7} we use the following Lemma.

\begin{lem}
\label{mainthm3a}
Let $f:(0,\infty)\to \R$ be such that
\be\label{hypf2}
f(s)\ge c_0 s-c_1,\quad s>0,
\ee
for some $c_0, c_1>0$. Then:

(i) There exists $R_0=R_0(c_0,n)>0$ 
such that any
positive solution of $-\Delta u =f(u)$ in~$\overline B_{2R_0}(x_0)$ satisfies
$$\int_{B_{R_0}(x_0)}u(x)\,dx\le C=C(c_0,c_1,n).$$

(ii) The problem
$$
\left\{\begin{array}{llll}
\hfill -\Delta u&=&f(u),& x\in\Rnp
\vspace{1mm} \\
\hfill u&=&0,& x\in\partial\Rnp
\end{array}
\right.
$$
does not admit any solution $u\in C^2(\overline\Rnp)$ such that $u_{x_n}>0$ and $u_{x_nx_n}\ge 0$ in $\overline\Rnp$.
\end{lem}

\begin{proof} 
(i) Take $R_0>0$ so that the first eigenvalue $\lambda_1(2R_0)$ of the Dirichlet Laplacian on $B_{2R_0}$ satisfies
$\lambda_1(2R_0)=c(n)(2R_0)^{-2}=c_0/2$.
The conclusion then follows by using assumption \eqref{hypf2} and testing $-\Delta u =f(u)$ in $B_{2R_0}$ with the corresponding first eigenfunction $\varphi_1$, which is such that $\varphi_1\ge \delta(R_0)>0$ in $B_{R_0}$, by the Harnack inequality.
\smallskip

(ii) Assume for contradiction that $u\in C^2(\overline\Rnp)$ is a solution
such that $u_{x_n}>0$ and $u_{x_nx_n}\ge 0$. Let $R_0$ be the number from (i).
Our assumptions imply that
$$\eta:=\inf \bigl\{ u_{x_n}(x',0);\ x'\in\R^{n-1},\ |x'|<R_0\bigr\}>0.$$
Since $u_{x_nx_n}\ge 0$, it follows that
$$u(x',y)\ge \eta y\quad\hbox{ for all $y>0$ and $x'\in\R^{n-1}$ such that $|x'|<R_0$.}$$
Hence by (i), for all $s>R_0$,
$$C\ge \int_{B_{R_0}(se_n)}u(x)\,dx \ge \omega_nR_0^n\eta (s-R_0),\quad s>R_0.$$
But this is a contradiction, for sufficiently large $s$.
\end{proof}

The hypothesis \eqref{hypf2} is clearly satisfied by any nonnegative convex function such that  $f(0)=0$ and $f\not\equiv0$. For instance, if $f(t_0)>0$ for some $t_0>0$, we have
$$
f(t)\ge \frac{f(t_0)}{t_0}\, t- f(t_0), \quad t>0,
$$
because  $(f(t)-f(0))/t$ is nondecreasing in t.
Theorem~\ref{mainthm7} is proved.
\bigskip

{\bf Acknowledgements.}
The first author is funded by CNPq grant 427056/2018-7, and FAPERJ grant E-26/203.015/2017.
This work was partially done during a visit of the second author to the Mathematics Departement of the
Pontifícia Universidade Cat\'olica do Rio de Janeiro. He thanks this institution for the hospitality.

\end{document}